\documentclass[12pt]{amsart}
\usepackage{a4}
\usepackage{amssymb}
\usepackage{amsmath}
\usepackage{amsthm}
\usepackage{amstext}
\usepackage{amscd}
\usepackage{latexsym}
\usepackage{graphics}
\theoremstyle{plain}
\newtheorem{thm}{Theorem}[section]
\newtheorem{lemma}[thm]{Lemma}

\newtheorem{cor}[thm]{Corollary}
\newtheorem{prop}[thm]{Proposition}
\theoremstyle{definition}
\newtheorem{rmk}[thm]{Remark}

\newtheorem{defn}[thm]{Definition}
\newtheorem{question}[thm]{Question}
\newtheorem{conjecture}[thm]{Conjecture}

\newcommand{\Prod}{\prod}

\newcommand{\sC}{{\mathcal C}}

\newcommand{\sG}{{\mathcal G}}

\newcommand{\sO}{{\mathcal O}}

\newcommand{\HH}{{\mathbb H}}

\newcommand{\Q}{{\mathbb Q}}
\newcommand{\R}{{\mathbb R}}

\input{xy}
\xyoption{all}

\begin{document} 

\title[Representation equivalent arithmetic
lattices]{Commensurability and representation \\equivalent 
arithmetic lattices} 

\author[C. Bhagwat]{Chandrasheel Bhagwat$^1$}
\address{Indian Institute of Science Education and Research, Pune-
  411021, India} \email{cbhagwat@iiserpune.ac.in} 
\thanks{$^1$ Research grant under the INSPIRE Faculty Award [IFA- 11MA -05]}.
\author[S.Pisolkar]{Supriya Pisolkar} \address{
School Of Mathematics, Tata Institute Of fundamental research, Homi
Bhabha Road, Mumbai- 400005, India}
\email{supriya@math.tifr.res.in} 

\author[C.S.Rajan]{C.S.Rajan}
\address{School of mathematics , Tata institute of fundamental
research, Homi Bhabha road, Mumbai- 400005, India}
\email{rajan@math.tifr.res.in} 

\date{}

\begin{abstract} Gopal Prasad and Rapinchuk defined a
notion of weakly commensurable lattices in a semisimple group, and gave a
classification of weakly commensurable Zariski dense subgroups.
A motivation was to classify pairs of locally symmetric
spaces isospectral with respect to the Laplacian on functions. For
this, in higher ranks,  they assume the validity of Schanuel's
conjecture. 

We observe that if we use the notion of representation
equivalence of lattices, then Schanuel's conjecture can be
avoided. Further, the results are applicable in a $S$-arithmetic setting.

We introduce a new relation `characteristic equivalence'
on the class of arithmetic lattices,
stronger than weak commensurability. This simplifies 
the arguments used in \cite{PR} to deduce commensurability type
results.
\end{abstract}

\maketitle

\section{Introduction}\label{intro}
Let $M$ be a compact, connected Riemannian manifold. The spectrum of the
Laplace-Beltrami operator  acting on the space of smooth functions on
$M$, the collection of its eigenvalues  counted with
(finite) multiplicity, is a discrete weighted subset of the
non-negative reals. Define two compact connected Riemannian
manifolds $M_1$ and $M_2$ to be {\em isospectral on functions} or just
{\em isospectral},  
if the spectra of the Laplace-Beltrami operator
acting on the space of smooth functions on   $M_1$ and $M_2$ coincide.

The inverse spectral problem is to recover the properties of the Riemannian
manifold $M$ from a knowledge of the spectrum.   
It is known, for example, 
that the spectra on functions determines the dimension, volume and
the scalar curvature of $M$. 

Milnor constructed the first examples in the context  of flat
tori of non-isometric compact Riemannian manifolds
which are isospectral on functions. 
When the spaces are compact hyperbolic surfaces, such examples
were initially constructed by  Vigneras \cite{V}. In analogy
with a construction in arithmetic, Sunada gave a general method for
constructing pairs of isospectral spaces \cite{S}. 

In many of these constructions, the manifolds are quotients by
finite groups of a fixed Riemannian manifold. The question arises
whether isospectral manifolds are indeed commensurable, i.e., have a
common finite cover. In the context of 
Riemannian locally symmetric spaces this question has been studied 
by various authors (\cite{R, CHLR,
  PR, LSV}) assuming that the spaces are isospectral  for the 
Laplace-Beltrami operator
acting on functions. Gopal Prasad and A. S. Rapinchuk address this question
in full generality, and get commensurability type results for
isospectral, compact locally symmetric spaces. For this when the 
locally symmetric spaces are of rank at least two, they have to
assume the validity of Schanuel's conjecture on transcendental
numbers. 

In this note, we consider this question 
assuming  a stronger hypothesis
that the lattices defining the locally symmetric spaces 
are representation equivalent rather than isospectral on
functions. 
This allows us to obtain similar conclusions as in
\cite{PR} 
for representation equivalent lattices,
without invoking Schanuel's conjecture, and
also extend the application to representation equivalent
$S$-arithmetic lattices. In the process, we introduce a new relation of
characteristic equivalence of lattices, stronger than weak
commensurability. This stronger hypothesis  
helps in simplifying  some of the  arguments used in
\cite{PR}. 

\section{Representation equivalence of lattices} The Fourier analysis
for the circle group $S^1$ can be studied in two ways: either,  as
expanding a function in terms of  the eigenfunctions of the Laplace
operator, or via characters of the topological group $S^1$. In the
context of Riemannian locally symmetric spaces the spectrum can also
be studied  in terms of representation theory of the isometry group of
the universal cover. 

Let $G$ be a locally compact, unimodular topological group and
$\Gamma$ be a uniform lattice in $G$.  Let $R_{\Gamma}$ denote the
right regular  representation of $G$ on the space
$L^2(\Gamma\backslash G)$ of square integrable functions with respect
to the projection  of the Haar measure on the space $\Gamma\backslash
G$:
\[(R_{\Gamma}(g)f)(x) = f(xg), \quad f \in L^2(\Gamma\backslash G),
\quad g, x \in G.\]
As a $G$-space, $L^2(\Gamma \backslash G)$ breaks up as a 
(Hilbert) direct sum of irreducible unitary representations of $G$,
\[ L^2(\Gamma\backslash G)\simeq \widehat{\bigoplus}_{\pi \in \hat{G}} m(\pi,
\Gamma) ~\pi,\] where $\hat{G}$ is the unitary dual of $G$
parametrizing isomorphism classes of irreducible, unitary
representations of $G$, and $ m(\pi, \Gamma)$ is the  (finite)
multiplicity with which  an element $\pi\in \hat{G}$ occurs in
$L^2(\Gamma\backslash G)$. Define the {\em representation spectrum} of
a uniform lattice $\Gamma \subset G$ to be the map $\pi \mapsto m(\pi,
\Gamma)$ giving the multiplicity $m(\pi, \Gamma)$ with which an
irreducible unitary representation $\pi$ of $G$ occurs in $
L^2(\Gamma\backslash G)$.

\begin{defn}\label{repeq}
 Let $G$ be a locally compact topological group and $\Gamma_1$ 
and $\Gamma_2$ be two uniform lattices in $G$. The lattices $\Gamma_1$ and 
$\Gamma_2$ are said to be {\em representation equivalent in $G$} if 
\[L^2(\Gamma_1\backslash G) \cong L^2(\Gamma_2\backslash G)\]
as $G$-spaces.
\end{defn}
 
The relevance of this notion to spectrum is provided by the following
generalization of Sunada's criterion for isospectrality \cite{DG}: 

\begin{thm} Let  $G$ be a locally compact topological  group $G$ 
which acts on a Riemannian manifold $M$. Let $\Gamma_1, ~\Gamma_2$ be  
representation equivalent uniform lattices in $G$. Suppose  $G$ acts on 
a Riemannian manifold $M$, such that $\Gamma_1, ~\Gamma_2$ act
properly discontinuously and freely on $M$ with compact quotients. 
Then the Riemannian manifolds $\Gamma_1\backslash M$ 
and  $\Gamma_2\backslash M$ with respect to the induced metric from
$M$ are strongly isospectral; in particular, they are isospectral on
$p$-forms for all $p$.
\end{thm}

The concept of strong isospectrality  is defined in \cite{DG} as
having the same spectrum for any natural (in the sense of Epstein and
Stredder) elliptic differential operator with positive definite
symbol. A plausible alternate definition is as follows: suppose two
compact oriented Riemannian manifolds $M, ~N$ are isospectral on
functions. Then it is known that their dimensions are equal, say of
dimension $d$.   The Riemannian metric gives a reduction of structure
group of the tangent bundle to the orthogonal group $SO(d)$. Given a
representation $\tau$ of $SO(d)$, this defines two metrized vector
bundles on $M$ and $N$ respectively. A Laplace type operator
(elliptic, self-adjoint, non-negative) can be defined on the space of
smooth sections of these bundles. For strongly isospectral, we require
that for any $\tau$ as above, these Laplace operators have the same
spectrum.  For example,  one can consider the   spectrum of the
Hodge-deRham Laplacians acting on the space of smooth $p$-forms of a
oriented compact Riemannian manifold.  

Suppose  $M=G/K$ is a noncompact 
Riemannian symmetric space, where $G$ is a noncompact semisimple Lie
group and $K$ is a maximal compact
subgroup of $G$. Let $\Gamma$ be a uniform torsion-free lattice
in $G$. To an irreducible representation $\tau$ of $K$
there is associated an automorphic vector bundle $E_{\tau}$ on the 
quotient space $\Gamma\backslash G/K$. The above theorem implies that
if the lattices are representation equivalent, then the spectra of the
Laplace operators on the smooth sections of $E_{\tau}$ are equal.

\begin{rmk}  
 In \cite{P},  Pesce has proved that
 the converse of the generalized Sunada Criterion holds in the case of
 $G = {\rm Isom}(\HH^n)$, where $\HH^n$ is the hyperbolic $n$-space
 with constant sectional curvature $-1$.  
However, in the general context of locally symmetric spaces, 
the converse to the generalized Sunada criterion is not known, i.e,
whether isospectrality for all automorphic vector bundles as above
yields representation equivalence. 

For compact hyperbolic surfaces $X$,  it is known that the spectrum on
functions determines the representation equivalence of the lattice
$\pi_1(X, x_0)\subset PSL(2,\R)$. This prompts the following question:

\begin{question}
Will it be true that for 
compact quotients of non-compact Riemannian symmetric spaces,  the
spherical spectrum (the restriction of the representation
spectrum to the class of spherical representations of $G$)  determines the
representation class of the lattice in the group of isometries
(\cite{BR})? More generally, will this be true if we just look at the
spectrum of the Laplacian on functions? 
\end{question}
\end{rmk}

\subsection{Arithmetic lattices}
We will have the following notations and assumptions for the rest of
this paper: 
\begin{description}
\item[H1] $\sG$ is a connected absolutely almost simple algebraic group
defined over a number field $K$. 

\item[H2]  $S$ is  a finite set of places of
$K$ containing the archimedean places at which $\sG$ is isotropic. Let
$S^i$ denote the subset of places of $S$ at which $\sG$ is isotropic.  

\item[H3] There is at least one place  $v\in S$ at
which $\sG$ is isotropic. 

\end{description}

A subgroup $\Gamma$
of $\sG(K)$ is said to be $(\sG, K, S)$-arithmetic (or just
arithmetic) subgroup, 
if it is commensurable with 
$\sG(\sO_K(S))=\sG(K)\cap GL_n(\sO_K(S))$,  
where $\sO_K(S)$ is the set of $S$-integers in $K$ and we consider
$\sG$ as embedded in  $GL_n$ over $K$ for some $n$. 

  Denote by $\sG_S$
the locally compact group, 
\[ \sG_S=\prod_{v\in S} \sG(K_v),\]
where given a place $v$
of $K$, $K_v$ denotes the completion of $K$ at $v$. 
  
There is an embedding of the arithmetic subgroup $\Gamma\subset \sG_S$,
which is well defined upto complex conjugation at the complex places
of $K$.  By results of Borel, Harishchandra, Godement and
Tamagawa, this defines an arithmetic
lattice $\Gamma\subset \sG_S$, which is Zariski dense in $\sG$. 

Suppose  $\sG_1, ~\sG_2$ are algebraic groups as above, and 
assume further that  are anisotropic. Then the
lattices $\Gamma_i$ are cocompact in $\sG_{i, S_i}$ for $i=1,~2$. 
We define two   $S$-arithmetic subgroups  $\Gamma_1\subset \sG_1(K_1),
~\Gamma_2\subset \sG_2(K_2)$ to be   {\em topologically 
representation equivalent} if
there exists an isomorphism $\phi:\sG_{1,S_1}\to \sG_{2,S_2}$ of topological groups
such that 
\[L^2(\phi(\Gamma_1)\backslash \sG_{2,S_2}) 
\cong L^2(\Gamma_2\backslash  \sG_{2,S_2})\]
as $ \sG_{2,S_2}$-spaces.

\begin{rmk} By theorems of Freudenthal and Borel-Tits \cite{BT}, it is
  known that any abstract homomorphism of adjoint Lie groups as above is
  automatically continuous. Hence in the definition of representation
  equivalence we could have just required that there is an abstract
  isomorphism between the ambient groups, requiring that the image of the
  lattice $\Gamma_1$ is again a lattice (so that representation
  equivalence makes sense). 
\end{rmk}

Denote by $\sG\to \overline{\sG}$ the isogeny to the 
adjoint group corresponding to $\sG$. For a subgroup $\Gamma$ of
$\sG(K)$, $\overline{\Gamma}$ will denote the image in $\overline{\sG}(K)$.  

Define two arithmetic subgroups $\Gamma_1\subset \sG_1(K_1),
~\Gamma_2\subset \sG_2(K_2)$ to be {\em commensurable}, if there are
isomorphisms $\sigma: {\rm Spec} K_2\to {\rm Spec} K_1$ and 
$\phi: ~{^{\sigma}\overline{\sG}}_1\to \overline{\sG}_2$, 
where the superscript $\sigma$ denotes twisting the group scheme
$\sG_1$ by $\sigma$. In particular, the image
$\phi(\overline{\Gamma}_1)$ considered as a subgroup of
$\overline{\sG}_2(K_2)$ and  $\overline{\Gamma}_2$ will be
commensurable subgroups.  

\subsection{Main Theorem}
Inspired by the work of Gopal Prasad and Rapinchuk, our aim now is to
obtain commensurability type results for representation equivalent
arithmetic lattices.  Working with representation equivalence of 
arithmetic lattices rather than
isospectrality on functions of the corresponding locally symmetric
space, allows us to avoid invoking the validity of Schanuel's
conjecture (see Conjecture \ref{schanuel}) on
transcendental numbers:  
\begin{thm} \label{main}
Let 
 $\sG_1$ (resp. $\sG_2$)  be anisotropic 
  algebraic groups defined 
over a number field $K_1$ (resp. $K_2$).  Let $S_1$  (resp. $S_2$) 
be a finite set of places  of
  $K_1$ (resp. $K_2$). Assume that for $i=1, ~2$, $(K_i, \sG_i, S_i)$
  satisfy hypothesis {\bf H1-H3}. 

Let  $\Gamma_1\subset \sG_1(K_1)$ (resp. $\Gamma_2\subset
\sG_2(K_2)$) be $S_1$ (resp. $S_2$)-arithmetic subgroup of  $\sG_1$
(resp. $\sG_2$). 

Suppose that the lattices $\Gamma_1\subset \sG_{1,S_1}, ~
\Gamma_2\subset \sG_{2,S_2}$ are topologically representation
equivalent. 

Then the following hold:

\begin{enumerate}
\item The groups $\sG_1$ and $\sG_2$ are of the same geometric
  type, i.e.,  $\overline{\sG}_1\times \overline{K}\simeq 
\overline{\sG}_2\times
  \overline{K}$.

\item The fields $K_1$ and $K_2$  are Galois conjugate. 

\item There exists an isomorphism   $\sigma: K_1\to K_2$ such that 
 the set of isotropic places  coincide: $S_1^i= \sigma^*(S_2^i)$.  

\item If $\sG_1$ is not of type $A_n,  ~D_{2n+1}, ~E_6 ~(n>1)$,
  then the lattices $\Gamma_1$ and $\Gamma_2$ are commensurable, i.e., 
$\overline{\sG}_1\simeq 
\overline{\sG}_2$ over $K$. 

\item In any topologically representation equivalence class of
  arithmetic lattices, there are only finitely many commensurability 
classes of  arithmetic lattices. 

\end{enumerate}
\end{thm}

Part (1) of the above theorem, follows immediately from the definition
of topologically representation equivalent lattices. The existence of
an isomorphism between $\sG_{1,S_1}$ and $\sG_{2,S_2}$ gives an
isomorphism at the level of Lie algebras. By assumption, at any place
$v_1\in S_1$ (resp. $v_2\in S_2$), the Lie algebra of
$\sG_1(K_{1,v_1})$ (resp. $\sG_2(K_{2,v_2})$) is simple. Hence (1)
follows.  

\begin{rmk} Since the lattices $\Gamma_i$ are uniform for $i=1, 2$,  any element
  belonging to $\Gamma_i$ is semisimple. 
\end{rmk}

\begin{rmk}
The first instance of this theorem was established by A. Reid
\cite{R}, who showed that the spectrum of the Laplacian
on functions of an arithmetic
compact hyperbolic surface associated to a quaternion division
algebra defined over a totally real number field determines the
underlying number field and the division algebra. 

For a compact Riemannian manifold $M$, denote by $L(M)$ the subset of
$\R$ consisting  of lengths of closed geodesics in $M$. 
Two Riemannian manifolds $M_1$ and $M_2$ are
said to be length commensurable (resp. length isospectral) if
$\Q L(M_1)=\Q L(M_2)$ (resp. $L(M_1)=L(M_2)$).   The starting
point of the proof of Reid's theorem is to use the Selberg trace
formula to conclude that two compact hyperbolic surfaces are
isospectral if and only if their length spectrums coincide.

Reid also proved that the complex length spectrum (length together
with the holonomy of the closed geodesic) of a compact, arithmetic
hyperbolic three manifold determines the commensurability class of the
manifold. It can be seen from the trace formula 
that the complex length spectrum determines
the representation equivalence class of the lattice. 
Working with only the length spectrum,  Chinburg, Hamilton,
Long and Reid showed  in \cite{CHLR} that length
commensurable hyperbolic three manifolds are commensurable. 
 
These results were vastly generalized by Gopal Prasad and A. Rapinchuk
(\cite{PR}). First, using results of Duistermaat, Guillemin, Kolk and
Varadarajan (\cite{DG, DKV}), Prasad-Rapinchuk-Uribe-Zelditch show
that if two compact, Riemannian locally symmetric spaces of
nonpositive sectional curvature are isospectral for the
Laplace-Beltrami operator on functions then they are length
commensurable (see Theorem 10.1 in \cite{PR}). 

Prasad and Rapinchuk define a notion of weak commensurability of
lattices: 
\begin{defn}
 Let $G_1$ and $G_2$ be two semi-simple groups defined over a field
 $F$ of characteristic zero. Two Zariski dense 
subgroups $\Gamma_i$ of $G_i(F)$, for $i=1,2$ are said to be 
 weakly commensurable if given any  element of infinite order 
$\gamma_1 \in \Gamma_1$ (resp.  $\gamma_2 \in \Gamma_2$) there exists
 an element of infinite order $\gamma_2 \in \Gamma_2$  
(resp.  $\gamma_1 \in \Gamma_1$)
such that the subgroup of $\bar{F}^*$ 
generated by the eigen values of $\gamma_1$  (resp.  $\gamma_2$)
(in a faithful representation of 
$G_1$)  intersects nontrivially the subgroup generated by the
eigenvalues of an element  $\gamma_2$ (resp.  $\gamma_1$).
\end{defn}
Prasad and Rapinchuk show (\cite[Section 10]{PR}) 
that length commensurable arithmetic lattices are
weakly commensurable. For this, when 
the locally symmetric spaces are of rank greater than one, 
they assume the validity of Schanuel's conjecture: 

\begin{conjecture}[Schanuel] \label{schanuel}
If $z_1, \cdots, z_n$ are $\Q$-linearly independent complex numbers,
then the transcendence degree over $\Q$ of the field generated by 
 $$z_1, \cdots, z_n, e^{z_1}, \cdots, e^{z_n}$$
 is at least $n$. 
\end{conjecture}

From the notion of weak
commensurability of lattice, 
using methods from arithmetic theory of algebraic groups, they obtain
results on commensurability, in particular the conclusions of  Theorem
\ref{main}.   

The use of representation equivalence instead of isospectrality on
functions allows us to bypass the use of Schanuel's conjecture in the
higher ranks. The proof of Theorem \ref{main} is an application of the
Selberg trace formula and the ideas and methods
given in \cite{PR}. 
\end{rmk}

\begin{rmk} An initial motivation for this paper was to extend the 
results of A. Reid \cite{R} to the context of $S$-arithmetic groups. 
An advantage of working with the representation
  theoretic spectrum, is that the notion applies even when there is no
 Riemannian geometric interpretation. This allows us to consider
$S$-arithmetic lattices. 
\end{rmk}

\begin{rmk}  
Examples of representation equivalent lattices which are not
commensurable have been given by Lubotzky, Samuels and Vishne
\cite{LSV}. It would be interesting to know whether such examples can 
be constructed in  the exceptional cases given in \cite[Section 9]{PR}, where
commensurability fails.

\end{rmk}
\section{Element-wise conjugate lattices}
\begin{defn}\label{topec} Let $G$ be a locally compact
group and $\Gamma_1, ~\Gamma_2$ be lattices in
$G$.  The lattices $\Gamma_1$ and  $\Gamma_2$ are said
to be elementwise conjugate in $G$ if for any
element $\gamma_1\in \Gamma_1$ (resp. $\gamma_2\in \Gamma_2$) there
exists an element $\gamma_2\in \Gamma_2$ (resp. $\gamma_1\in
\Gamma_1$) such that $\gamma_1$ and $\gamma_2$ are conjugate in
$G$. 
\end{defn}

An application of the Selberg trace formula for compact quotients
yields the following theorem: 

\begin{thm}\label{thm:retoec}  Let $G$ be a locally compact
groups 
 and $\Gamma_1, ~\Gamma_2$ be uniform
lattices in $G$. Suppose the lattices $\Gamma_1$ and
$\Gamma_2$  are representation equivalent. Then they are
elementwise conjugate. 
\end{thm} 

\begin{cor}\label{cor:retoec} 
With notation as in Theorem \ref{main}, suppose that the lattices 
$\Gamma_1\subset \sG_{1,S_1}, ~
\Gamma_2\subset \sG_{2,S_2}$ are topologically representation
equivalent by an isomorphism $\phi: \sG_{1,S_1}\to \sG_{2,S_2}$. Then
$\phi(\Gamma_1)$ and $\Gamma_2$ are elementwise conjugate in
$\sG_{2,S_2}$.
\end{cor}

\subsection{Selberg trace formula}
We recall the Selberg trace formula for uniform lattices \cite{W}.
 Let $f$ be a continuous, 
compactly supported function on $G$. The convolution
operator $R_{\Gamma}(f)$ on $L^2(\Gamma \backslash G)$ is defined by, 
\[
R_{\Gamma}(f)(\phi)(x) = \int \limits_{G}
f(y)~R_{\Gamma}(y)(\phi)(x)~ \text{d}\mu(y), 
 \]
where $\mu$ is an invariant Haar measure on $G$. It is known that 
 $R_{\Gamma}(f)$ is of trace class. 

Let $[\gamma]_{G}$ (resp. $[\gamma]_{\Gamma}$) be the conjugacy class
of $\gamma$ in $G$ (resp. in $\Gamma$). Let $\left[\Gamma\right]$
(resp.  $\left[\Gamma \right]_G$) be the set of conjugacy classes in
$\Gamma$ (resp. the $G$-conjugacy classes of elements in $\Gamma$).
For $\gamma \in \Gamma$, let $G_{\gamma}$ be the centralizer of
$\gamma$ in $G$. Put $\Gamma_{\gamma} = \Gamma \cap G_{\gamma}$. It
can be seen that $\Gamma_{\gamma}$  is a lattice in $G_{\gamma}$ and
the quotient $\Gamma_{\gamma} \backslash G_{\gamma}$ is compact.
Since $G_{\gamma}$ is unimodular, there exists a $G$-invariant measure
on $G_{\gamma} \backslash G$, denoted by $\text{d}_{\gamma}x$. 
After
normalizing the measures on $G_{\gamma}$ and $G_{\gamma} \backslash G$
appropriately and rearranging the terms on the right  hand side of
above equation, we get :

\begin{equation}\label{1} \text{tr}(R_{\Gamma}(f)) \quad = \quad
\sum \limits_{[\gamma]\ \in \ \left[\Gamma \right]} \text{vol}
(\Gamma_{\gamma} \backslash G_{\gamma}) \ \int\limits_{G_{\gamma}
\backslash G} f(x^{-1} \gamma x)\ \text{d}_{\gamma}x
\end{equation}

\begin{equation*}\label{geom} = \sum \limits_{[\gamma]\ \in\
[\Gamma]_G}a(\gamma,\Gamma)\ O_{\gamma}(f)
\end{equation*} where $O_{\gamma}(f)$ is the orbital integral of $f$
at $\gamma$ defined  by,  
\[ O_{\gamma}(f) = \int\limits_{G_{\gamma} \backslash G} f(x^{-1}
\gamma x)\ \text{d}_{\gamma}x.\] Here \[ a(\gamma, \Gamma) = \sum
\limits_{[\gamma']_\Gamma \ \subseteq\ [\gamma]_G}\ \text{vol} \
(\Gamma_{\gamma'} \backslash G_{\gamma'}). \]  

If $\gamma$ is not
conjugate to an element in  $\Gamma$,  we define $a(\gamma,
\Gamma)=0$.  \\

\noindent Let $\pi$ be an irreducible unitary representation of
$G$.  Denote by $\chi_{\pi}(f)$ the distributional character of $\pi$
given by, 
\[ \chi_{\pi}(f)= {\rm Trace}(\pi(f)).\]
The trace of $R_{\Gamma}(f)$ on the
spectral side can be written as an absolutely convergent series as,
\begin{equation}\label{rep} \text{tr}(R_{\Gamma}(f)) = \sum
\limits_{\pi\ \in \ \widehat{G}} m(\pi,\Gamma)~\chi_{\pi}(f)
\end{equation}
 
Hence from (\ref{1}) and (\ref{rep}), we obtain the Selberg trace
formula: 

\begin{equation}\label{STF1} \sum \limits_{\pi\ \in \ \widehat{G}}
m(\pi,\Gamma)~ \chi_{\pi}(f) = \sum \limits_{[\gamma]\ \in\
[\Gamma]_G}a(\gamma,\Gamma)\ O_{\gamma}(f). 
\end{equation}

\subsection{Proof of Theorem \ref{thm:retoec}}

We prove a few lemmas before giving the proof of Theorem \ref{thm:retoec}.

\begin{lemma}\label{finite}  
Let $G$ be a locally compact
topological group and $\Gamma$ be a uniform lattice in $G$. Let $U$
be a relatively compact subset of $G$. Then the set 
\[ A_{U}  =  \left\{\ [\gamma]_G : \gamma \in \Gamma\ \text{and}\
[\gamma]_G \cap U \neq \emptyset\ \right\}\] is finite.  
\end{lemma}  

\begin{proof} Since the quotient $\Gamma \backslash G$ is compact,
there exists a relatively compact subset  $D$ of $G$ such that $G
= \Gamma D$. Let $x \in G$ be such that $x^{-1} \gamma x \in U$ for
some $\gamma \in \Gamma$.  Write $x = \gamma'. \delta$ where $\gamma'
\in \Gamma$ and $\delta \in D$.  Hence ${\gamma'}^{-1} \gamma \gamma'
\in D U D^{-1}$ which is relatively compact in $G$. Hence
${\gamma'}^{-1} \gamma \gamma'  \in D U D^{-1} \cap \Gamma$ which is a
finite set.

\end{proof}

\begin{lemma}\label{U} Let $\Gamma_1$ and $\Gamma_2$ be uniform
lattices in $G$. Let $\gamma_1 \in \Gamma_1$.  Then there exists a
relatively compact open set $U$ containing $\gamma_1$ such that \[ U
\cap [\gamma]_{G} = \emptyset\] whenever $\gamma \in \Gamma_1 \cup
\Gamma_2 ~\text{and}~ [\gamma_1]_{G} \neq [\gamma]_{G}.$ 
\end{lemma}

\begin{proof}
Easily follows from Lemma \ref{finite}.
\end{proof} 

\begin{proof}[Proof of Theorem \ref{thm:retoec}]
By comparing the Selberg trace formula (\ref{STF1}) 
for the lattices $\Gamma_1$ and 
$ \Gamma_2$  in $G$, we get for any compactly supported continuous
function $f$ on $G$, 
\begin{equation*}\label{trf1} \sum \limits_{\pi\ \in \ \widehat{G}}
 \left[m(\pi,\Gamma_1) ~- ~ m(\pi,\Gamma_2) \right]~ \chi_{\pi}(f) = 
\sum \limits_{\substack{[\gamma]\  \in \
[\Gamma_1]_G \cup [\Gamma_2]_G}} \left[a(\gamma,\Gamma_1) -
a(\gamma,\Gamma_2)\right] \ O_{\gamma}(f). 
\end{equation*} 
Since the lattices $\Gamma_1, ~\Gamma_2$ are representation
equivalent in $G$, the left side is identically zero in the above
equation. 

Suppose $\gamma_{1} \in \Gamma_1$ is not conjugate to
any  element of $\Gamma_2$ in $G$.
Choose $U$ as in Lemma \ref{U}, and a positive 
function $f$ supported on $U$. For such $f$, we have that the orbital 
integral $O_{\gamma}(f)$ vanishes 
whenever $[\gamma]_{G} \neq [\gamma_{1}]_{G}$. Further 
$O_{\gamma_1}(f)$ is non-zero. 

It follows that  all terms on the right hand side of the above
equation vanish except that
corresponding to $[\gamma_1]_{G}$. Consequently, $a(\gamma_1,
\Gamma_1)~ O_{\gamma_1}(f) = 0$. Since both these quantities are
non-zero by definition, we arrive at a contradiction. Hence the
lattices $\Gamma_1$ and $\Gamma_2$ are elementwise 
conjugate in $G$. 
\end{proof}

\section{Characteristic equivalence of lattices}
Corollary \ref{cor:retoec} assures us that the elements in two
topologically representation equivalent arithmetic
lattices are elementwise conjugate (upto an
isomorphism) in some large group, for instance in the group of complex
points of the algebraic group. 
In particular this implies  that the
lattices $\Gamma_1$ and $\Gamma_2$ are weakly commensurable.  
Theorem \ref{main} follows now from the results proved by Gopal Prasad
and A. Rapinchuk (\cite{PR}[Theorems 1 to 5]). 

The conclusion of
Corollary  \ref{cor:retoec} is stronger than the notion of  weak
commensurability.  However, it does not seem easy to go directly
from elemenwise conjugacy  to commensurability
results, for example, to obtain Theorem \ref{localrank}. 
This leads us to define a new relation on the class of
arithmetic lattices,  stronger than weak  commensurability, which we
call as characteristic equivalence.  We 
rephrase the property of elementwise conjugacy in terms of
characteristic polynomials. This notion allows us to directly invoke results
from the arithmetic theory of algebraic groups and simplify
the arguments deducing commensurability type results from weak
commensurability given in \cite{PR}.

Let $\sG$ be an algebraic group defined over a number field
$K$. Consider the adjoint action $Ad$ of $\sG$ on its Lie algebra
$L(\sG)$. Given a semisimple element $\gamma \in \sG(L)$ for an extension field
$L$ of $K$, and any field $M$ containing $L$, denote by 
$P(Ad_{\sG}(\gamma), x)$ the characteristic polynomial of $Ad(\gamma)$
acting on $L(\sG)\otimes_K M$. The characteristic  polynomial is independent
of the extension field $M$, and has coefficients in $L$. In
particular, if $\Gamma\subset \sG(K)$ is an arithmetic lattice, and if $v$ is
any place of $K$, then the characteristic polynomials coincide, 
\[P(Ad_{\sG}(\gamma), x)=P(Ad_{\sG(K_v)}(\gamma_v), x),\]
where $\gamma\in \Gamma$ and by $\gamma_v$ we denote its image in
$\sG(K_v)$.  

Note that the characteristic polynomial is also independent of the
isogeny class of $\sG$: given $\gamma \in \sG(K)$, then 
\[P(Ad_{\sG}(\gamma), x)=P(Ad_{\overline{\sG}}(\overline{\gamma}), x),\]
where $\overline{\gamma}$ denotes the image of $\gamma$ in the adjoint
group $\overline{\sG}(K)$. 

The characteristic polynomial is also independent upto isomorphisms: 

\begin{lemma}\label{indcharpoly}
Let $\sG_1, ~\sG_2$ be simple algebraic groups defined over an
algebraically closed field $F$, and let $\theta:\sG_1, \to \sG_2$ be
an isomorphism defined over $F$. Suppose $t$ is a semisimple element
in $\sG_1(F)$. Then 
\[  P(Ad_{\sG_1}(t, x))= P(Ad_{\sG_2}(\theta(t), x)).\]
\end{lemma}
\begin{proof}
Let $T_1$ be a maximal torus in $\sG_1$ containing $t$. The
eigenvalues of $Ad_{\sG_1}(t)$ are $0$ with multiplicity equal to the
rank of $\sG_1$ and $\alpha(t)$ where $\alpha$ runs over the roots of
$L(\sG_1)$ with respect to $T_1$. If $X_{\alpha}$ is a root vector
corresponding to the root $\alpha$, then 
\[
Ad(\theta(t)(d\theta(X_{\alpha}))=d\theta(Ad(t)X_{\alpha})=\alpha(t)X_{\alpha}.\]
Hence the eigenvalues of $\theta(t)$ are the same as $t$,
and this proves the lemma. 

\end{proof}

The topological elementwise conjugacy of the lattices $\Gamma_1$ and
$\Gamma_2$ given by Corollary  \ref{cor:retoec} yields the following key
proposition stating an equality of characteristic polynomials with
respect to the adjoint representation:
\begin{prop}\label{prop:chareq}
With assumptions as in Theorem \ref{main}, there exists a locally
compact field $F$ and embeddings $\iota_1:K_1\to F, ~\iota_2: K_2\to
F$, and a topological automorphism $\sigma$ of $F$ such that  
 given
any  element $\gamma_1\in \Gamma_1$  (resp. $\gamma_2\in \Gamma_2) $
there exists an element 
$\gamma_2\in \Gamma_2$  (resp. $\gamma_1\in \Gamma_1$) 
such that the characteristic polynomials coincide, 
\[ \sigma(P(Ad_{\sG_1}(\gamma_1), x))=P(Ad_{\sG_2}(\gamma_2), x).\]
\end{prop}

\begin{proof}
Let  $v_1\in S_1^i$ be an isotropic place of $\sG_1$. The group $
{\sG}_1(K_{1,v_1})$ is a non-compact normal subgroup of $\sG_{1,
  S_1}$. Hence there exists an isotropic place $v_2\in S_2^i$ for
$\sG_2$, such that the projection to ${\sG}_2(K_{2,v_2})$ of the image
$\phi({\sG}_1(K_{1,v_1}))$ is a non-compact normal subgroup $N$ of  
${\sG}_2(K_{2,v_2})$. Since $\sG_2$ is absolutely almost simple, $N$
is Zariski dense in $\sG_2$.

By Theorem A of Borel-Tits (\cite{BT}),  
there is a continuous homomorphism  $\sigma: K_{1,
  v_1}\to K_{2, v_2}$ such that the map 
${\sG}_1(K_{1,v_1})\to  \overline{\sG}_2(K_{2,v_2})$ is induced by
an algebraic morphism between the base changed group schemes,  
\[ ^{\sigma}({\sG}_1\times K_{1,v_1})\to 
\overline{\sG}_2\times K_{2,v_2}, \]
where the superscript $\sigma$ denotes twisting the group scheme
$\sG_1$ by $\sigma$. 
This map yields an isomorphism of algebraic
groups at the adjoint level. 

By Corollary  \ref{cor:retoec}, given any element $\gamma_1\in
\Gamma_1$
(resp. $\gamma_2\in \Gamma_2$) there exists an element $\gamma_2\in
\Gamma_2$ (resp. $\gamma_1\in \Gamma_1$)
such that the element $\overline{\phi({\gamma}_1)}$
(resp.  $\overline{\gamma}_2)$  is conjugate in $\overline{\sG}_2(K_{2,v_2})$ to
$\overline{\gamma}_2$ (resp. 
$\overline{\phi({\gamma}_1)}$).

Let $F=K_{2, v_2}$ and $\iota_2: K_2\to F$ be the natural
embedding. The restriction of $\sigma$ to $K_1$ gives an embedding
$\iota_1$ of $K_1$ into $F$. By Lemma \ref{indcharpoly} and the
remarks preceding it, we have
\[ 
P(Ad_{\sG_2}(\gamma_2), x)=P(Ad_{\sG_2}(\phi(\gamma_1), x)
= \sigma(P(Ad_{\sG_1}(\gamma_1), x).
\]
This proves the proposition. 
\end{proof}

We now show  Part (2) of  Theorem \ref{main}, 
that the fields of definition of the arithmetic lattices
are conjugate: 

\begin{proof}[Proof of Part (2) of  Theorem \ref{main}] 
In the notation of the proof of the foregoing proposition,
let  $K_1'=\iota_1(K_1)$. Consider
the group  
$\Gamma_1':={\iota_1(\overline{\Gamma}_1)}$ as
an arithmetic lattice of the group 
$ \overline{\sG}_1'={^{\iota_1}\overline{\sG}}_1$
defined over the number field $K_1'$. 
We have  an algebraic isomorphism $\theta: 
\overline{\sG}_1'\times F\to \overline{\sG}_2\times F$ defined over $F$ 
of the groups base changed to $F$. Further $\theta({\Gamma}_1')$ and
$\overline{\Gamma}_2$ are elementwise conjugate in $ \overline{\sG}_2(F)$

By a theorem of Vinberg  as given in Lemma 2.6 of
\cite{PR}, it follows that the fields generated by ${\rm Trace}
({\rm Ad}(\gamma))$ for $\gamma$ belonging to  $\Gamma_1'$
(resp. $\Gamma_2$) generate the field of definition $K_1'$
(resp. $K_2$) of the ambient  group $\overline{\sG}_1'$
(resp. $\overline{\sG}_2$). Hence $K_1'=K_2$ and this proves Part (2)
of Theorem \ref{main}. 

\end{proof}

Henceforth,  we will assume upto twisting the group scheme $\sG_1$ by a field
automorphism $\sigma:K_1\to K_1'$, that $K:= K_1=K_2$ and both the
group schemes $\sG_1$ and $\sG_2$ are defined over the same number
field $K$.

We now define a notion of characteristic equivalence of lattices:
\begin{defn}
Let 
 $\sG_1$ (resp. $\sG_2$)  be 
  algebraic groups defined respectively 
over a number field $K$.  Let $S_1$  (resp. $S_2$) 
be a finite set of places respectively of
  $K$. Assume that for $i=1, ~2$, $(K, \sG_i, S_i)$
  satisfy hypothesis {\bf H1-H3}. 

Let  $\Gamma_1\subset \sG_1(K)$ (resp. $\Gamma_2\subset
\sG_2(K)$) be $S_1$ (resp. $S_2$)-arithmetic subgroup of  $\sG_1$
(resp. $\sG_2$). 

We say that $\Gamma_1$ and $\Gamma_2$ are {\em characteristically
equivalent} if given
any semisimple element $\gamma_1\in \Gamma_1$ (resp. $\gamma_2\in \Gamma_2$) 
there exists a semisimple element 
$\gamma_2\in \Gamma_2$ (resp. $\gamma_1\in \Gamma_1$) 
such that the characteristic polynomials coincide, 
\[ P(Ad_{\sG_1}(\gamma_1), x)=P(Ad_{\sG_2}(\gamma_2), x).\]
\end{defn}

\begin{lemma} \label{isogtori}
In the definition of characteristic equivalence, we can
  further assume that the  tori given  by the identity component of
  the algebraic subgroup generated by $\gamma_i$
  in $\sG_i$ ($i= 1, 2$) are isogenous. 
\end{lemma}
\begin{proof} Since the algebraic groups are absolutely almost simple,
  the isogeny class of the  tori given  by the identity component of
  the algebraic subgroup generated by $\gamma_i$
is determined by the element $Ad_{\sG_i}(\gamma_i)$ for $i=1,
2$. Identifying the Lie algebras with $K^N$ as vector spaces over $K$,
 the lemma follows. 
\end{proof}

We now establish Theorems 1 to 5 of \cite{PR} 
under this stronger hypothesis of characteristic equivalence of
lattices: 

\begin{thm}\label{main2}
Let 
 $\sG_1$ (resp. $\sG_2$)  be 
  algebraic groups defined respectively 
over a number field $K$.  Let $S_1$  (resp. $S_2$) 
be a finite set of places respectively of
  $K$. Assume that for $i=1, ~2$, $(K, \sG_i, S_i)$
  satisfy hypothesis {\bf H1-H3}. 

Let  $\Gamma_1\subset \sG_1(K)$ (resp. $\Gamma_2\subset
\sG_2(K)$) be $S_1$ (resp. $S_2$)-arithmetic subgroup of  $\sG_1$
(resp. $\sG_2$). 

Suppose that $\Gamma_1$ and $\Gamma_2$ are characteristically
equivalent lattices. Then the following holds: 

\begin{enumerate}
\item The groups $\sG_1$ and $\sG_2$ are of the same geometric
  type, or one of them is of type $B_n$ and the other is of type
  $C_n$. 

\item The set of isotropic  places
  $S_{1}^i$ and $S_{2}^i$ coincide.

\item Assume further that $\sG_1$ and $\sG_2$ are of the same
  geometric type. 
 If $\sG_1$ is not of type $A_n,  ~D_{2n+1},~(n>1) ~E_6$,
  then the lattices $\Gamma_1$ and $\Gamma_2$ are commensurable. 

\item In any characteristic equivalence class of
  arithmetic lattices, there are only finitely many commensurability 
classes of  arithmetic lattices. 

\end{enumerate}

\end{thm}
Theorem \ref{main2} combined with Proposition \ref{prop:chareq} gives a
proof of Theorem \ref{main}. 

\section{Proof of Theorem \ref{main2}}
Let $\sG$ be a connected, absolutely almost simple algebraic group
defined over $K$. Let $T$ be a maximal $K$-torus in $\sG$. 
Denote by $\Phi_T$ the root system of $\sG$ with respect to $T$, and
by $W(\Phi_T)$ the Weyl group of $\Phi_T$. 
Let $L$ be the  splitting field of $T$.
There exists a natural injective homomorphism 
$\theta_{T}: {\rm Gal} (L/K)\to {\rm Aut} (\Phi_T)$. 

For the proof of Theorem \ref{main2}, we need the following theorem on the
existence of irreducible tori (\cite{PR2}[Theorem 1]):

\begin{thm}\label{irrtori}
Let $\sG$ be a connected, absolutely almost simple algebraic group
defined over a number field $K$. Suppose $v$ is a place of $K$ and
$T_v$ is a maximal $K_v$-torus of $\sG$. Then there exists a
$K$-torus $T$ of $\sG$ such that
it is conjugate to $T_v$ by an element of $\sG(K_v)$. Further, the
image of $\theta_T$ contains the Weyl group $W(\Phi_T)$. In
particular, $T$ is an irreducible, anisotropic maximal $K$-tori of
$\sG$. 
\end{thm}
The proof of this theorem is based on a theorem of A. Grothendieck
that the variety of maximal tori is rational, and based on this a
theorem  of
V. E. Voskresenskii showing 
that the Galois group of the splitting field of the 
generic maximal tori  contains the Weyl group. 

\begin{cor}\label{corirrtori} 
With notation as in Theorem \ref{irrtori}, let
$\Gamma$ be  a $S$-arithmetic
lattice in $\sG$ and  $v\in S^i$. Assume further that $T_v$
is an isotropic torus.  Then there exists an element
$\gamma\in \Gamma$ which generates $T$ over $K$. 
\end{cor}
\begin{proof}
By \cite[Theorem 5.12]{PlR},  there exists non-torsion elements
in $T_1(\sO_K(S))$.  Since  
$ \Gamma \cap T(\sO_K(S))$ is of finite index in
$T(\sO_K(S))$, there exists a non-torsion element 
$\gamma \in \Gamma \cap T(\sO_K(S))$. Since $T$ is irreducible,
$\gamma$ will generate $T$ over $K$. 
\end{proof}

\begin{proof}[Proof of Part (1) of Theorem \ref{main2}]
The equality of the characteristic polynomials with respect to the
adjoint representation implies that the dimensions of the Lie algebras
are equal. If the algebraic groups involved are not of type $B_6,
~C_6$ or $E_6$, then the geometric type is determined by the
dimension of the Lie algebra. 

For the proof of Part (1) in this exceptional case, 
we argue as in proof of Theorem
1 in \cite[page 130]{PR}: by Corollary \ref{corirrtori}, 
choose a torus $T_1$ and an element
$\gamma_1\in \Gamma_1$ which generates $T_1$ over $K$.  
By characteristic equivalence, there exists an element
$\gamma_2\in \Gamma_2$ having the same characteristic polynomial as
$\gamma_1$. By Lemma \ref{isogtori}, we can further assume that the
tori $T_1$  generated
by $\gamma_1$, and the tori $T_2$  given by the identity component of
the diagonalizable subgroup generated by  $\gamma_2$ are isogenous over $K$.  

Let $L$ be the  splitting field of $T_1$ (equivalently of
$T_2$). By Theorem \ref{irrtori}, the image of $\theta_{T_1}$ contains
the Weyl group $W(\Phi_{T_1})$. We can assume that the geometric
type of $\sG_1$ is of type either $B_6$ or $C_6$. In this case, all
automorphisms of $\Phi_{T_1}$ are inner, and the cardinality of ${\rm
  Gal}(L/K)$ is thus equal to
 $|W(\Phi_{T_1})|$. From the injectivity
of the map $\theta_{T_2}$, we see that  $|W(\Phi_{T_1})|$ divides the
cardinality of  ${\rm Aut} (\Phi_{T_2})$. But the cardinality of 
$W(B_6)$ is $2^{10}3^25$, whereas the cardinality of ${\rm Aut}(E_6)$
is given by $2^73^45$. This implies that
$\sG_2$ cannot be of type $E_6$. 

\end{proof}

\begin{proof}[Proof of Part (2) of Theorem \ref{main2}]
This is the analogue of Theorem 3 of \cite{PR}, and we 
follow the proof as given in \cite[page 139]{PR} of this theorem. If $v\in
S_1^i$ is a place where $\sG_1$ is isotropic, choose a maximal split
$K_v$-torus $T_{1,v}$ of $\sG_1$. By Theorem \ref{irrtori}, there exists 
a $K$-irreducible anisotropic maximal $K$-torus $T_1$ of $\sG_1$ such that
it is conjugate to $T_{1,v}$ by an element of $\sG_1(K_v)$. Since $T_1$ 
is anisotropic the quotient $T_{1,S_1}/T_1(\sO_K(S_1))$ is compact 
 where $T_{1,S_1} = \Prod_{v \in S_1} T_1(K_v)$.
This implies that the
quotient $T_1(K_v)/\sC$ is also compact, 
where $\sC$ is the closure of $T_1(\sO_K(S_1))$
in $T_1(K_v)$. Since $T_1$ is $K_v$-isotropic,
 $\sC$ is noncompact. The closure of $Ad(T_1(\sO_K(S_1)))$ inside 
$GL_N(K_v)$ will also be noncompact.  Since $T_1(\sO_K(S_1))$ is 
a finitely generated abelian group consisting of semisimple elements,
it can be simultaneously diagonalised over $\overline{K}_v$. If the
eigenvalues of every element in $T_1(\sO_K(S_1))$ is a $v$-adic unit,
then this implies that the closure of $T_1(\sO_K(S_1))$ is compact,
contradicting our earlier conclusion. Since
$ \Gamma_1 \cap T_1(\sO_K(S_1))$ is of finite index in
$T_1(\sO_K(S_1))$, there exists an element 
$\gamma_1 \in \Gamma_1 $ 
such that   at least one eigenvalue of 
$Ad_{\sG_1}(\gamma_1) \in GL_N$ is not a $v$-adic unit.
By assumption there exists an element $\gamma_2 \in \Gamma_2$ which is
characteristic equivalent to $\gamma_1$.

If $v \not\in S_2^i$, then the
closure of the subgroup $\sG_2(\sO_K(S_2))$ in $\sG_2(K_v)$ is
compact. But this
implies that all the eigenvalues of 
$Ad_{\sG_2}(\gamma_2)$ are $v$-adic units. 
This yields a contradiction and hence  $S_{1}^i \subset S_{2}^i$. By symmetry 
we get $S_{1}^i= S_{2}^i$.
\end{proof}

\begin{rmk} It is known that weak approximation holds for the tori
  constructed in Theorem \ref{irrtori} (see \cite{PR2}). One could
  have also used this fact to give a slight variation of the above
  argument. 
\end{rmk}

We now prove Theorem 6.2 of \cite{PR}. 
It is the basic input needed to prove Parts (3) and (4) of Theorem
 \ref{main2}. 

\begin{thm} \label{localrank}
With hypothesis as in Theorem \ref{main2}, for any place
$v $ of $K$, 
 $${\rm rk}_{K_v}\sG_1 = {\rm rk}_{K_v}\sG_2$$ 
\end{thm}
\begin{proof}
Let $T_{1,v}$  be  
a maximal $K_v$-split torus 
of $\sG_1$, and choose a $K$-torus $T_1$ and an element $\gamma_1\in
\Gamma_1$ as in Corollary \ref{corirrtori}. 

By the characteristic equivalence of $\Gamma_1$ and $\Gamma_2$, 
 there exists an element $\gamma_2 \in \Gamma_2$ 
for which there is an equality of  characteristic polynomials
\[ P(Ad_{\sG_1}(\gamma_1), x)=P(Ad_{\sG_2}(\gamma_2), x).\] 
This implies that the elements $Ad_{\sG_1}(\gamma_1)$ and 
$Ad_{\sG_2}(\gamma_2)$ considered as elements in $GL_N/K$ are
conjugate, and hence generate isomorphic diagonalizable subgroups
 over $K$. Let $T_2$ be
the tori given by the identity component of the subgroup 
generated by $\gamma_2$. 
We have,
$${\rm rk}_{K_v}\sG_1= {\rm rk}_{K_v}T_1 = {\rm rk}_{K_v}T_2\leq \rm{rk}_{K_v}\sG_2$$
By symmetry, this proves the theorem. 
\end{proof}

The proofs of Part (3) and (4) of Theorem \ref{main2} follow as in
page 147-148 of \cite{PR}. For the sake of completeness, we give a
brief outline of the proof. 

\begin{proof}[Proof of Part (3)  of Theorem \ref{main2}]
If the geometric type is of type $D_{2n}, ~(n>2)$ (resp. $D_4$)
this is proved in \cite{PR3} (resp. \cite{G}). 

If the geometric
type is not of $A, ~ D$ or $E_6$ type, the equality of local ranks
implies that $\overline{\sG}_{1,v}\simeq \overline{\sG}_{2,v}$ for any
place $v$ of $K$. For archimedean places, this 
follows from classification results \cite{T}. 
For a non-archimedean place, this follows from the
fact that there can be at most two possible forms for the
adjoint group. To see the latter fact, we
observe that the centre $Z$ of the simply connected cover of $\sG$
is a subgroup of $\mu_2$, where $\sG$ is not of type $A, ~D, ~E_6$. 
 From the equality of the Galois cohomology
groups, we get that $H^1(K_v, \overline{\sG})\simeq H^2(K_v,
Z)$, which can be identified with a subgroup of the 
$2$-torsion in the Brauer group
of $K_v$.  Since this is of cardinality two, and the outer
automorphism group is trivial, this implies that there are at most two
forms of $\overline{\sG}$ for any non-archimedean place $v$. Hence an
equality of ranks over $K_v$ implies that the forms are isomorphic.

 Now Part (3) of
Theorem \ref{main2} follows from the Hasse principle, viz., the
injectivity of the localization map,
$$ H^1(K, \overline{\sG})\to \bigoplus_v  H^1(K_v, \overline{\sG})$$
where $v$ runs over all places of $K$. 
\end{proof}
  
\begin{proof}[Proof of Part (4)  of Theorem \ref{main2}]
From Theorem \ref{localrank}, it can be seen by a Chebotarev density argument 
(\cite[Theorem 6.3]{PR}
that the minimal splitting field $L_i$ over which $\sG_i$ becomes the
inner form of a split group for $i=1,~2$ coincide. Moreover, the set of
places $V_i$ at which $\sG_i$ is not quasi-split coincide.

Fixing the geometric type, say a split form $\sG_0$ over $K$ of adjoint type,
the groups are parametrized by cocycles $c \in H^1(K, {\rm Aut}(\sG_0))$. 
Consider the exact sequence, 
\[ 1\to \sG_0\to  {\rm Aut}(\sG_0)\to  {\rm Out}(\sG_0)\to 1.\]
This yields an exact sequence, 
\[ H^1(K, {\sG}_0)\to  H^1(K, {\rm Aut}(\sG_0)) \to H^1(K,
{\rm Out}(\sG_0)).\]
The condition that the group  becomes an inner form over $L$ implies that this
cocycle lies in the image of $H^1(G(L/K), {\rm Out}(\sG_0))$ which is
a finite group.  Now the
required finiteness follows from the finiteness of the Hasse
principle, i.e., the kernel of the localization map, 
$$ H^1(K, \overline{\sG}_0)\to \bigoplus_{v\not\in V}  H^1(K_v,
\overline{\sG}_0), $$
where $V=V_i$ is a fixed finite set of places of $K$. 
\end{proof}

\begin{rmk} It is further deduced \cite[Theorem 6]{PR}, 
that if $\sG_1$ and
  $\sG_2$ have the same geometric type satisfying  the hypothesis of
  Theorem \ref{localrank}, then the Tits indices are equal at all
  places of $K$.  
\end{rmk}

\begin{rmk} It is possibly more appropriate to call the notion of
  characteristic equivalence given out here as {\em weakly
    characteristic equivalence}, since we are not taking into account
  multiplicities. A notion of multiplicity will be to count the number
  upto $\Gamma$-conjugacy of the set of elements in $\Gamma$ which have
  the same characteristic polynomial with respect to the adjoint
  representation. 
It would then be interesting
  to know whether characteristically equivalent lattices (counted with
  multiplicities) are
  (topologically) representation equivalent. 

It is also clear that stably conjugate elements i.e., conjugate in
$\sG(\overline{K})$ will have the same characteristic
polynomials. This yields another possible definition of characteristic
equivalence, but it is in the conjugacy in the Lie algebra version
that we can directly relate to the underlying arithmetic of the
ambient group $\sG$.  

It is possible to consider modifications of the 
concept of characteristic  equivalence, say more generally 
  on the class of subgroups not necessarily arithmetic lattices: 
 for example one can consider
  the equality of the characteristic polynomials on `big' subsets, like
  subgroups of finite index, or Zariski open, or 
even some kind of Hilbertian sets. 

Yet another relation that can be imposed is to define two 
  lattices to be  trace equivalent if the set of traces of elements with
  respect to the adjoint representation coincide for the two
  lattices. 

It would be interesting to know whether these properties
would  imply commensurability results. To conclude commensurability type
results will require analogues of Theorems \ref{irrtori} or
\ref{localrank}. 
\end{rmk}

\begin{rmk} It is clear that characteristically equivalent lattices
  are weakly commensurable. Examples have been given in \cite[Sections
  6 and 9]{PR} of weakly commensurable lattices which are not
  commensurable.  It would be interesting to know whether these
  examples give   characteristically equivalent lattices. 

\end{rmk}

\noindent{\bf Acknowledgement.} We thank the referee for useful comments.

\end{document}